\documentclass[11pt,a4paper]{article}
\usepackage[utf8]{inputenc}
\usepackage[T1]{fontenc}
\usepackage{amsmath}
\usepackage{amssymb}
\usepackage{amscd}
\usepackage{amsthm}
\usepackage{proof}
\usepackage{color}

\newtheorem{theorem}{Theorem}[section]
\newtheorem{corollary}[theorem]{Corollary}

\newtheorem{prop}[theorem]{Proposition}
\newtheorem{rem}[theorem]{\bf Remark}
\newtheorem{definition}[theorem]{Definition}

\newtheorem{example}[theorem]{\bf Example}

\theoremstyle{definition}

\newcommand{\Hmm}[1]{\leavevmode{\marginpar{\tiny%
$\hbox to 0mm{\hspace*{-0.5mm}$\leftarrow$\hss}%
\vcenter{\vrule depth 0.1mm height 0.1mm width \the\marginparwidth}%
\hbox to 0mm{\hss$\rightarrow$\hspace*{-0.5mm}}$\\\relax\raggedright #1}}}

\newcommand{\en}{{\cal E}}
\newcommand{\den}{{\dom}(\en)}
\newcommand{\dom}{\mathrm{dom}}
\newcommand{\inr}{\mathrm{Inr}}
\newcommand{\vol}{\mathrm{vol}}

\newcommand{\diam}{\mathrm{diam}}
\newcommand{\lao}{\lambda_0^D(\Omega)}
\newcommand{\lae}{\lambda_1^N}
\newcommand{\NN}{\mathbb{N}}

\newcommand{\cD}{\mathcal{D}}
\newcommand{\cP}{\mathcal{P}}
\newcommand{\cB}{\mathcal{B}}
\newcommand{\RR}{\mathbb{R}}
\newcommand{\cex}{\RR^X}
\newcommand{\inromega}{R_\Omega}
\newcommand{\supp}{\mathrm{supp}}
\newcommand{\set}[2]{\bigl\{#1\bigm|#2\bigr\}}
\newcommand{\varn}[1]{\left\|#1\right\|_{\mathcal{V}}}
\newcommand{\bigslant}[2]{{\raisebox{.2em}{$#1$}\left/\raisebox{-.2em}{$#2$}
\right.}}
\title{Universal lower bounds for Laplacians on weighted
graphs}
\author{D.~Lenz\footnote{ Mathematisches Institut, Friedrich Schiller
Universit{\"a}t Jena, 07743 Jena, Germany, daniel.lenz@uni-jena.de},
P.~Stollmann\footnote{Fakult\"{a}t f\"{u}r Mathematik,  Technische
Universit\"{a}t Chemnitz, D-09107 Chemnitz, Germany,
P.Stollmann@mathematik.tu-chemnitz.de } }

\begin{document}

\maketitle

\begin{abstract}
We discuss optimal lower bounds  for eigenvalues of Laplacians on
weighted graphs. These bounds are formulated in terms of the
geometry and, more specifically, the inradius of subsets of the
graph. In particular, we study the first non-zero eigenvalue in the
finite volume case and the first eigenvalue of the Dirichlet
Laplacian on subsets that satisfy natural geometric conditions.
\end{abstract}

\section{Introduction}
The first main objective of this article is to give an overview over
recent results that deal with lower bounds on the first non--zero
eigenvalue of Laplacians on graphs. These lower bounds involve the
inradius of subsets. This topic certainly has a long history and key
arguments have been rediscovered over and again. It seems that the
classical paper by Beurling and Deny \cite{BD-58} is the first
source for the key step, a fact that we just recently became aware of and that
has
gone unnoticed in the literature. The lower bounds we mentioned can
be seen as discrete analogs of Poincar\'e or Payne-Weinberger
inequalities \cite{PayneWeinberger-60}, the latter both being of
fundamental importance for analysis on continuum spaces.

A second point to be emphasized it that the bounds in question are
\emph{universal}. By this  we mean inequalities that do not require
additional assumptions on the graph, be it combinatorial assumptions
e.g., regularity or finiteness or  geometrical assumptions, e.g.,
curvature restrictions. This is in particular motivated by
applications in mathematical physics, where random (sub-) graphs
have seen a great deal of interest. By the very nature of randomness
these subgraphs tend to not satisfy further regularity assumptions.

Since we deal with Laplacians of general weighted graphs our
treatment includes several important classes considered in different
mathematical fields, in particular combinatorial Laplacians and
normalized Laplacians in the usual setting of combinatorial graphs.
The estimates we derive in the finite volume case are quite easy to
state and to prove and determine a lower bound that basically
depends on the diameter (or the inradius) as well as on  the volume.
Simple examples as well as more involved ones show that these
estimates are optimal. Yet, there is even  a more conceptual way of
showing optimality that was put forward in the recent \cite{LSS-TPI}
based on the following observation: in a  first step towards a
Poincar\'e inequality one estimates the variation of a function in terms
of the energy and an appropriate metric. Such an estimate was called topological
Poincar\'e inequality in the last named paper and might maybe better
be  called a \emph{metric Poincar\'e inequality}. This crucial bound
had already been proved in the above mentioned \cite{BD-58} and does not
involve
the underlying measure. It gives rise to a Poincar\'e inequality with respect
to a different metric,
sometimes called resistance metric. Varying the measure one can show that
the corresponding Poincar\'e inequality is optimal; see Section
\ref{sec-mpi} for details.

The third main objective of the present article is a generalization of
the already quite general set--up considered in \cite{LSS-arxiv}. We
briefly explain the situation here and refer to Section
\ref{sec-infinite-volume} for the full picture. A non--zero lower
bound for the first Neumann eigenvalue can in general not be
expected in the infinite volume case. However, the lowest  Dirichlet
eigenvalue $\lambda^D_0(\Omega)$ of the Laplacian on the open subset
$\Omega$ can be strictly positive, even if the volume of $\Omega$ is
infinite. The reader may think of a strip in euclidean space for a
continuum analog. Such an estimate can be deduced with the help of a
Voronoi decomposition for $\Omega$ that exhibit suitable relative
boundedness properties, of which finite inradius is the property
that replaces finite diameter and a second condition can be thought
of as a relative finite volume property. In euclidean space, a
Voronoi decomposition can easily be written down. In general
weighted graphs, however, the situation is much more complicated and
certain topological properties of the graph are necessary, see
Proposition \ref{prop-existencevoronoi} and Example \ref{ex-voronoi}
for details. Our new result, Theorem \ref{infinite-D}, now holds in
full generality, in particular in situations where Voronoi
decompositions are no longer available. The main new idea of the
proof is really simple: for any particular given function we work on
a taylor--made finite subgraph to verify the appropriate energy
estimate.

While it is not the topic of this article a  comment on Cheeger
inequalities may be in order. Cheeger inequalities provide  lower
bounds on the minimum of the spectrum of Laplacians in terms of
geometry viz the isoperimetric constant. For graphs such
inequalities were first discussed by Alon and Milman \cite{AM} (for
finite graphs) and by Dodziuk \cite{Dod} (for infinite graphs). The
work \cite{Dod} features a somewhat unsatisfactory additional
dependence on the measure. This could later be removed in the
situation of normalized Laplacians by Dodziuk and Kendall in
\cite{DK}. A satisfactory answer for general Laplacians on infinite
graphs was only recently established by Bauer, Keller and
Wojciechowski \cite{BKW}. It involves intrinsic metrics. The
corresponding bounds do not require any finiteness condition on the
inradius. Hence, they apply in situations where the bounds discussed
below do not give any useful information. On the other hand in
situations involving subsets of lattices, where our bounds apply,
they tend to be better than bounds via Cheeger estimates,  see
\cite{LSS-arxiv}, Section 6, for comparison.

\section{Set--up and main results}
In this section we present our set up and the main results. These
results connect geometric data and spectral data
of a graph, where our concept of graph is a very
general one. In the presentation of this section we take special care
to introduce the spectral data with as little technical effort as
possible. A thorough discussion of the operator and form theoretic
background is given in the last section.

\bigskip

A weighted graph is a triple  $(X,b,m)$ satisfying the following
properties:
\begin{itemize}
 \item  $X$ is an arbitrary  set, whose elements are  referred to as
\textit{vertices};
 \item $b:X\times X\to [0,\infty)$  is a symmetric  function with $b(x,x)=0$
for all $x\in
 X$.
 \item  $m:X\to (0,\infty)$ is  a  function on the vertices.
\end{itemize}
An element $(x,y)\in X\times X$ with $b(x,y)>0$ is then  called an
\textit{edge}  and $b$ is denoted as \textit{edge weight}; The
positive function $m:X\to (0,\infty)$ gives a measure on $X$ of
which we think as a volume. In particular, we define
$$\vol (\Omega):=\sum_{x\in \Omega} m(x)$$
for $\Omega \subset X$.

A sequence of vertices  $\gamma=(x_0, ..., x_k)$ is called a
\textit{path}  from $x$ to $y$ if $x_0 = x$, $x_k = y$ and
$b(x_l,x_{l+1}) >0$ for $l=0,\ldots, k-1$. Throughout our graphs
will be assumed to be \textit{connected} i.e. to allow for a path
between arbitrary vertices.

A natural distance to consider is
$$
d(x,y):=\inf\{L(\gamma)\mid \gamma\mbox{  a path from }x\mbox{  to
}y\},
$$
where  the length $L(\gamma)$  of a path $\gamma$ is given by
$$
L(\gamma):=\sum_{j=0,...,k-1}\frac{1}{b(x_j,x_{j+1})} .
$$
Setting $d(x,x)=0$ we obtain a pseudo-metric, i.e., $d$ is symmetric
and satisfies the triangle inequality. Clearly, in this generality,
$d$ need not separate the points of $X$, a fact that is of no
importance for what follows.

We call the graph \textit{geodesic} if for any $x,y\in X$ there
exists a path $\gamma$ from $x$ to $y$ with $L(\gamma) = d(x,y)$.

We denote by
$$ U_r(x):= \{y\in X \mid d(x,y) < r\} \;\mbox{ and }\; B_r(x):= \{y\in
X \mid d(x,y) \le r\}$$ the  open and closed balls of radius $r$,
respectively. The \textit{diameter} of $X$ is given by
$$\diam (X) :=\sup\{d(x,y) : x,y\in X\}.$$

The positive function $m:X\to (0,\infty)$ together with the distance
$d$ gives the geometrical data of the space. For further details on
weighted graphs and their geometry as expressed by $d$ and related
metrics we refer to the recent studies \cite{Geo,GHKLW} as well as
the surveys \cite{Ke,KL-survey} and   the upcoming monograph
\cite{KLW}.

Spectral data are given in terms of the \textit{energy} $\en$
associated to the graph  defined by the quadratic form
$$
\en(f):=
\frac12 \sum_{x,y\in
X}b(x,y)(f(x)-f(y))^2 \mbox{  for  }f\in\cex,
$$
which may assume the value $\infty$ for the time being. The underlying Hilbert
space is
$$
\ell^2(X,m):=\{ f\in\RR^X\mid  \sum_{x\in X} f(x)^2m(x)<\infty\}
$$
with inner product and norm given by
$$\langle f,g\rangle =\sum_{x\in X} f(x) g(x) m(x) \mbox{ and }
\|f\| = \sqrt{\langle f, f\rangle}$$ respectively.

The Neumann Laplacian $H=H^N(X,b,m)$ is the self-adjoint operator
associated with $\en$ on its maximal domain in $\ell^2(X,m)$. We
refer the reader to the Appendix \ref{no-fear} for a precise definition
as well as further  details concerning forms, operators and all
that. The quantity we want to estimate in the case of  finite volume
and finite diameter is  the first non-zero eigenvalue $\lae$ of $H$.
This can be written down in variational terms as
$$
\lae(X)=\inf\{\en(f)\mid f\in\ell^2 (X,m) \mbox{ with }
\en(f)<\infty, \langle f,1\rangle = 0  \mbox{ and } \|f \|=1\}.
$$
Again, we refer to the Appendix for the fact that this is
the first non-zero eigenvalue of $H$ in the case of finite volume
and finite diameter. Readers who do not want to bother with
operators should just stick to the definition above.  No operator or
spectral theoretic technicalities enter the quite elementary proof
of our first main result, Theorem \ref{lb-fin-N}, which says that
$$
\lae\ge\frac{4}{\diam(X)\vol(X)} .
$$
Even though our proof partly uses known techniques, we did not find
any reference in this generality, see the discussion in Section
\ref{sec-finite-volume}. Let us point out, that the finite volume
estimate above can be thought of as a universal Poincar\'{e} or
spectral gap inequality that holds without any further restriction
on the weights. In particular, it is valid for combinatorial and
normalized Laplacians. Moreover, it is optimal.
This supports the feeling that the distance $d$ is rather well
suited for spectral geometry in a general setting where no further
information on the graph is available.

We also  deduce a lower bound for the Dirichlet Laplacian $H_\Omega$
for $\Omega\subset X$.  The latter is again defined in terms of
forms and the relevant functions in the form domain are supposed to
vanish on the complement $D:=X\setminus\Omega$. The spectral
quantity of interest is the first eigenvalue. It can be written down
in variational terms as
$$
\lao = \inf\{\en(f)\mid f\in\ell^2 (X,m) \mbox{ with }
\en(f)<\infty, \;  f =0 \mbox{ on   $D$ and $\|f\|=1$ } \}.
$$
It is easily seen (see the Appendix \ref{no-fear} again) that
$$
\lao= \min\sigma (H_\Omega) .
$$
(In the unlikely event that $\en  (f) = \infty$ for all $f$ with
support  contained in $\Omega$,  $\lao =\infty$ in accordance with
the usual conventions.)

Theorem \ref{lb-fin-D} says
$$
\lao\ge\frac{1}{\inromega\vol(\Omega)} ;
$$
with the  \textit{inradius} $\inromega$  given by
$$\inromega:=
\sup\{ r>0\mid \mbox{ there exists }  x\in\Omega \mbox{ such that }
U_r (x)\subset\Omega\} .
$$
Clearly, if  both $\vol(\Omega)$ and $\inromega$ are finite, this
gives a positive lower bound, otherwise the usual convention
$\infty^{-1}=0$ makes the statement valid but evident.

As will be seen below  the proof of Theorem \ref{lb-fin-N} and the
proof of Theorem \ref{lb-fin-D} are very similar in nature.  In this
context it may be elucidating to point out  diameter and inradius
are strongly related concepts.  In fact,  it is not hard to see that
$$\diam (X) = \sup\{ R_{X\setminus \{x\}} : x\in X\}.$$
So one may actually think of  both Theorem \ref{lb-fin-D} and
Theorem \ref{lb-fin-N} as giving  bounds in terms of inradius.
Indeed, it is possible to even  derive the bound given in Theorem
\ref{lb-fin-N} (up to the factor $4$)  from Theorem \ref{lb-fin-D},
compare discussion at the end of Section \ref{sec-finite-volume}.

Subsequently,  in Section \ref{sec-infinite-volume}, we are able to
give a lower bound on the Dirichlet Laplacian in the infinite volume
case. Of course, this can only be hoped for if $D:=X\setminus
\Omega$ is relatively dense i.e. that there exist an $R>0$ such that
any point in $\Omega$ has distance no more than $R$ to a point of
$D$. Note that this relative denseness can equivalently be seen as
finiteness of the inradius of $\Omega$. Indeed, the inradius of
$\Omega$ is nothing but the infimum over all possible such $R\geq
0$, see \cite{LSS-arxiv} for details. So, here again, the relevant
geometry enters via the inradius.

Moreover a relative volume estimate enters our analysis measured in
terms of
$$
\vol^\sharp[r]:= \inf_{s>r}\sup\{ m(U_s(x))\mid x\in X\}
$$
and we get the following result:
$$
\lao \ge \frac{1}{\inromega\vol^\sharp[\inromega]} .
$$
The proof combines a decomposition technique from \cite{LSS-arxiv}
with a new approximation argument. In fact, the decomposition into
Voronoi cells used in \cite{LSS-arxiv} needs some  geometric
properties of the underlying weighted graphs. Specifically, it was
established in  the latter reference for locally compact geodesic
weighted graphs.
%%%%%%

\section{Lower bounds in the finite volume
case}\label{sec-finite-volume} The proofs of both theorems in this
section go along similar lines. One main idea is an estimate that
relates the energy of a function $f$ and the maximal growth of $f$
over a certain distance. This estimate is given in the following
lemma. It has has been noted in various places and it seems the \cite{BD-58},
Remarque 3, p. 208 is the original source. We include the simple proof for
completeness reasons and later also interpret the estimate below as an
estimate between  different metrics, as done in the latter reference.

\begin{prop}[Basic proposition]\label{prop-basic}
Let $\gamma$ be a path from $x$ to $y$ and $f\in \RR^X$ with
$\mathcal{E}(f)< \infty$ be given. Then,
$$(f(x) - f(y))^2 \leq  L (\gamma) \mathcal{E} (f).$$
\end{prop}
\begin{proof}  Let $f$, $\gamma=(x_0, ... , x_k)$ be as above. Then,
 \begin{eqnarray*}
  \left(f(x) - f(y)\right)^2&=&\left(f(x_0)-f(x_k)\right)^2\\
&=&\left(\sum_{j=0}^{k-1}\sqrt{b(x_j,x_{j+1})}(f(x_j)-f(x_{j+1}))\frac{1}{\sqrt{
b(x_j ,x_{j+1})}}\right)^2\\
  &\le&
\sum_{j=0}^{k-1}b(x_j,x_{j+1})\left(f(x_j)-f(x_{j+1})\right)^2\sum_{j=0}^{k-1}
\frac{1}{
b(x_j,x_{j+1})}\\
  &\le& L(\gamma)\en(f).
  \end{eqnarray*}
\end{proof}

Now, we can directly derive the result on the Dirichlet case.

\begin{theorem}\label{lb-fin-D}
Let $(X,b,m)$ be as above.
Let  a non-empty $\Omega\subsetneq X$  be given and assume
$\vol(\Omega)<\infty$ and
$\inr(\Omega)<\infty$. Then
 $$
 \lao\ge \frac{1}{\inromega\vol(\Omega)} .
 $$
\end{theorem}
\begin{proof}[Proof of Theorem \ref{lb-fin-D}]
Let $R>\inromega$. Consider $x\in \Omega$ and $f\in \RR^X$ with
$\mathcal{E}(f)<\infty$ and  $f =0$ on $D=X\setminus\Omega$. By
definition of the inradius there is $x_0\in U_R(x)\setminus\Omega$.
In particular, there is a path $\gamma=(x_0,...,x_k)$ from $x_0$ to
$x=x_k$ of length at most $R$ and $f(x_0)=0$. Using the above
proposition, we get
$$
f(x)^2\le R\en(f) .
$$
Summing over $x\in\Omega$ and using that $R>\inromega$ was arbitrary we
conclude that
$$
\| f \|_{\ell^2}^2\le \inromega \vol(\Omega)\en(f) ,
$$
and hence, the assertion of the theorem follows.
\end{proof}

The proof of Theorem \ref{lb-fin-N} requires one more ingredient
borrowed from \cite{LSS-TPI}, where it was stated in the special
situation at hand:
\begin{prop}
Let $(Y,\cB,\mu)$ be a  finite measure space. Then,
for any bounded and measurable $f:Y\to\RR$ with $f\perp 1$:
$$
\| f \|_{2}^2\le \frac14 \sup_{x,y\in
Y}\left(f(x)-f(y)\right)^2 \mu(X) .
$$
\end{prop}

\begin{theorem}\label{lb-fin-N}
Let $(X,b,m)$ be as above. Assume  $\vol(X)<\infty$ and
$\diam(X)<\infty$. Then
 $$
 \lae(X)\ge \frac{4}{\diam(X) \vol(X)}.
 $$
\end{theorem}
\begin{proof} This directly follows  from Proposition \ref{prop-basic} and the
previous proposition.
\end{proof}

\begin{rem} (a) For finite combinatorial graphs, the lower bound is
a familiar bound and our proof follows known lines, compare Lemma
1.9 in \cite{Chu} and Lemma 2.4 in \cite{BCG} for related estimates as well
as \cite{Mohar}, estimate (4.1) in Theorem 4.2, p. 62, where it is attributed to
McKay.

(b)  Theorem 4.1 from \cite{BCG} gives that the lower bound is
optimal up to constants, i.e., it is achieved, up to constants by
balls in certain graphs.  As announced in the introduction, we will
present another approach to optimality in terms of variation of the
measure $m$ in Section \ref{sec-mpi}.
\end{rem}

As it is instructive we conclude this section by showing how a
(slightly  weaker) version of Theorem \ref{lb-fin-N} can easily be
derived from Theorem \ref{lb-fin-D}: We consider the situation of
Theorem \ref{lb-fin-N} and let $f\in \ell^2 (X,m)$ with
$\en(f)<\infty$, $\|f\| =1$ and $f\perp 1$ be given. Setting
$f_+:=\max\{f,0\}$ and $f_-:=\max\{-f,0\}$ we obtain the
decomposition $f = f_+ - f_-$. Clearly $\Omega_+:=\{x\in X : f(x)
>0\}$ and $\Omega_- :=\{x\in X : f(x) < 0\}$ are disjoint with $f_+$ vanishing
outside
$\Omega_+$ and $f_-$ vanishing outside of $\Omega_-$.  This easily
gives
\begin{eqnarray*}
\en(f_+,f_-) &= &\frac{1}{2}\sum_{x,y\in X} b(x,y) (f_+ (x) - f_+
(y)) (f_-(x) - f_-(y))\\
& =& -\sum_{x,y\in X} b(x,y) f_+ (x) f_-(y) \leq 0.
\end{eqnarray*}
Now, a direct computation involving Theorem \ref{infinite-D} gives
\begin{eqnarray*}
\en(f) &=& \en (f_+) + 2 \en (f_+,f_-) + \en(f_-)\\
&\geq & \en (f_+) + \en (f_-)\\
(\mbox{Theorem \ref{lb-fin-D}})\;\: &\geq
&\frac{\|f_+\|^2}{\vol(\Omega_+) R_{\Omega_+}} +
\frac{\|f_-\|^2}{\vol(\Omega_-) R_{\Omega_-}}\\
&\geq &
\frac{ \|f_+\|^2 + \|f_-\|^2}{\vol(X) \diam (X)}\\
&=& \frac{1}{\vol(X) \diam (X)}.
\end{eqnarray*}
Here, we used the obvious bounds $R_{\Omega_+}, R_{\Omega_-}\leq
\diam (X)$ and $\vol(\Omega_+),\vol(\Omega_-)\leq \vol(X)$ in the
penultimate step
%%%%%%

\section{Lower bounds for the Dirichlet Laplacian in the infinite volume
case}\label{sec-infinite-volume} In this section we consider
$H_\Omega$ again, this time without assuming that
$\vol(\Omega)<\infty$. Clearly, the estimate from Theorem
\ref{lb-fin-D} breaks down in this case. The basic idea, already
employed in \cite{LSS-arxiv}, is to decompose the large, infinite
graph $X$ into finite volume pieces on which the above estimate can
be used. The resulting energy estimates  can be summed up to give
the Poincar\'{e} type inequality. The way this was implemented in
\cite{LSS-arxiv} required strong assumptions on the underlying
geometry. Here we show how - based on the result of \cite{LSS-arxiv}
one can actually get rid of all these assumptions.

We start with a discussion of the relevant  decompositions. These
were  introduced in \cite{LSS-arxiv} and had already been used
in \cite{SSV-14} in a more restrictive set--up.
\begin{definition}\label{def-vor}
Let $(X,b,m)$ be as above and  $D\subset X$ non-empty. A
\emph{Voronoi decomposition} of $X$ with centers from $D$ is a
pairwise disjoint family $(V_p)_{p\in D}$ such that  following
conditions hold:
 \begin{enumerate}
  \item[(V1)] For each $p\in D$ the point $p$ belongs to  $ V_p$  and
  for all $x\in V_p$ there exists a path $\gamma$ from $p$ to $x$ that lies
   in $V_p$ and satisfies  $L(\gamma) = d(p,x)$.
  \item[(V2)] For each $p\in D$ and for all  $x\in V_p$ the inequality
  $d(p,x) \leq  d(q,x)$ holds  for any $q\in D$.
  \item[(V3)] $\bigcup_{p\in D}V_p=X$.
 \end{enumerate}
\end{definition}

Given a Voronoi decomposition with centers from $D$  one can obtain
a lower bound on the values of $Q(f)$ for $f$ vanishing on $D$  as
follows: Let $D\subset X$ be given and let $V_p$, $p\in D$, be an
Voronoi decomposition with centers from $D$. Assume that there
exists for each $p\in D$ a $c_p
>0$ with
$$ \frac{1}{2} \sum_{x,y\in V_p} b(x,y) (f(x) - f(y))^2 \geq c_p \sum_{x\in V_p}
f(x)^2 m(x)$$ for all $f$ with $f(p) =0$. Then, a direct summation
gives
$$\en(f) \geq (\inf_{p\in D} c_p ) \|f\|^2$$
for all $f$ which vanish on $D$ (compare  \cite{LSS-arxiv} as well).

We now turn to existence of Voronoi decompositions. This was shown
in \cite{LSS-arxiv} under the additional strong  geometric condition
of compactness of balls and this condition was shown to be
necessary. Here is the precise result.

\begin{prop}\label{prop-existencevoronoi}
Let $(X,b,m)$ be a connected graph such that all balls $B_r (x)$,
$x\in X$ and $r>0$,  are finite.  Assume that  $D\subset X$ is
non-empty and let $\Omega=X\setminus D$. Then there exists a Voronoi decomposition with centers
from $D$. Moreover, whenever $R=\inromega$ is finite then any
Voronoi decomposition $(V_p)_{p\in D}$ of $X$ with centers from $D$
has the property that $V_p\subset B_R(p)$ for all  $p\in D$.
\end{prop}

Let us give a simple example of a geodesic weighted graph that does
not allow a Voronoi decomposition:

\begin{example}\label{ex-voronoi}
Let $X:=\left(\NN\times \{ 0\}\right)\cup \{ (1,1)\}$ with weight
$b((n,0);(n+1,0))=2$ for $n\in\NN$,
$b((n,0),(1,1))=1+\frac{1}{n}$ for $n\in\NN$ and $b(x,y)=0$ else. Since none of
the points from
$D:=\NN\times \{ 0\}$ is closest to the point $(1,1)$, there is no Voronoi
decomposition of $X$ with
centers in $D$ in the sense of \cite{LSS-arxiv} in this case.
\end{example}

In the general  setting discussed in the present article  finiteness
of balls does not hold in general. So we can not directly appeal to
the proposition to get a Voronoi decomposition (and subsequent lower
bounds). Our method to circumvent this extra complication is quite
easy. For given $f$ vanishing on $D$, for which we want to estimate
the $\ell^2$--norm by the energy, we consider a finite subgraph of
$(X,b,m)$ on which a Voronoi decomposition exists by the result from
\cite{LSS-arxiv}. We get the following estimate and underline the
fact, that we have assumed no geometric restrictions apart from
connectedness!

\begin{theorem}\label{infinite-D}
Let $(X,b,m)$ be a connected weighted graph and $\Omega\subsetneq X$ such that
$\inromega<\infty$ and $\vol^\sharp[\inromega]<\infty$. Then
 $$
 \lao\ge \frac{1}{\inromega\cdot\vol^\sharp[\inromega]} .
 $$
\end{theorem}

\begin{proof}
Let $\varepsilon >0$, $R>\inromega$ and $f\in\ell^2(X)$ with $\| f\|=1$. We
have to show that
$$
1-\varepsilon \le R\cdot\sup\{ m(U_R(x))\mid x\in X\}\cdot \en(f)\qquad(*) .
$$
To this end, first note that there is a finite subset $X^0_f\subset X$ such that
$$
1-\varepsilon\le \| f\rvert_{X^0_f}\|^2 .
$$
We now enlarge $X^0_f$ suitably. First we add finitely many points to get a
subset $X^1_f\supset X^0_f$ so
that $(X^1_f,b\rvert_{X^1_f\times X^1_f})$ is connected. (This is possible as
$(X,b)$ is connected and so
there is a finite path connecting each of the finitely many pairs $(x,y)\in
X^0_f\times X^0_f$).

By assumption on $\Omega$ we know that
$$
X\subset \bigcup_{p\in D}U_R(p) .
$$
Therefore, we find a finite subset $D^0\subset D$ so that
$$
X^1_f\subset \bigcup_{p\in D^0}U_R(p) .
$$
By adding the points of $D^0$ as well as the points of finitely many suitably
chosen paths, we get
a finite subset $X^1_f\cup D^0\subset Y$ such that
$$
Y\subset \bigcup_{p\in D^0}U_R^Y(p) ,
$$
where the superscript $Y$ indicates that we are concerned with balls
in the induced subgraph $(Y,b\rvert_{Y\times Y}, m\rvert_Y)$. This
latter graph is finite, so in particular it satisfies the
assumptions of Proposition \ref{prop-existencevoronoi}, and  we get
a Voronoi decomposition with centers from $Y$. Now on each  $V_p$,
$p\in Y$,  of the Voronoi decomposition we can appeal to Theorem
\ref{lb-fin-D} to get a lower bound for the form on $V_p$, $p\in Y$.
Given this we can now proceed as discussed following the definition
 of Voronoi decomposition (with $Y$ instead of $D$)
 to obtain
$$ 1-\varepsilon\le
\| f\rvert_Y\|^2\le R\cdot\sup\{ m(U^Y_R)\mid y\in Y\}\cdot\en^Y(f)
.
$$
This implies $(*)$ which finishes the proof.
\end{proof}
%%%%%%
\section{Metric Poincar\'e inequalities and optimality of lower
bounds}\label{sec-mpi} In this section we put our results in the
context of Poincar\'{e} inequalities on graphs, see \cite{KS,HKSW}
for related recent results as well. This will be used to discuss
optimality.

\bigskip

We say the a pseudo-metric $p$ on $(X,b)$ \textit{satisfies a metric
Poincar\'e inequality}, provided
$$
\left( f(x)-f(y)\right)^2\le p(x,y)\en(f)\qquad\qquad\mbox{(MPI)}
$$
holds for all $f\in\RR^X$, $x,y\in X$. Clearly, our basic Proposition
\ref{prop-basic} says that
$d$  satisfies a metric Poincar\'e inequality.

The best constant $r(x,y)$ in (MPI) can easily be seen to satisfy
$r(x,y)=\rho(x,y)^2$, where
$$
\rho(x,y)=\sup\{ f(x)-f(y)\mid \en(f)\le 1\} .
$$
This latter  metric also goes back to \cite{BD-58}, where it was
called \emph{distance extr\'emale} and has later been rediscovered
in different contexts, e.g., in \cite{Davies-93}. While the validity
of the triangle inequality is evident for $\rho$, also $r$ itself
defines a pseudometric, a fact that is not so obvious. For details
(and further references) we refer to \cite{LSS-TPI}, in particular
Prop. 2.2.

We define the following seminorm on $\ell^\infty(X)$:
$$
\varn{f}:=\sup f -\inf f\mbox{  for  }f\in\ell^\infty(X).
$$
Moreover, we define
$$
\cD:=\set{f\in\RR^X}{\en(f)<\infty}
$$
and note the following equivalence:
\begin{prop}
 Let $(X,b)$ be a graph. Then the following statements are equivalent:
 \begin{itemize}
  \item[\rm{(i)}] A global variation norm Poincar\'e inequality holds, i.e.,
there is $c\ge 0$ such that
  for all $f\in\cD$:
  $$\varn{f}^2\le c \en(f)\qquad\qquad \mbox{{\rm (GVPI)}}
  $$
  \item[\rm{(ii)}] The diameter of $(X,b)$ w.r.t. $r$ is finite, i.e.,
  $$
  \diam_r(X)=\sup\set{r(x,y)}{x,y\in X}<\infty .
  $$
  \item[\rm{(iii)}] The graph  $(X,b)$ satisfies
  $$
  \cD\subset \ell^\infty(X) .
  $$
 \end{itemize}
Moreover, the best constant $C_P$ in {\rm (GVPI)} equals $\diam_r(X)$ and the
square norm $\| J\|^2$ of the inclusion map
$$
J: \bigslant{\cD}{\RR\cdot 1}\to \bigslant{\ell^\infty}{\RR\cdot 1}
$$
of quotients modulo the constant functions $\RR\cdot 1$.
\end{prop}
The equivalence of (i) and (ii) is rather obvious; the equivalence of (i) and
(iii) follows from a closed
graph theorem together with the observation:
\begin{eqnarray*}
 \varn{f}&=&\sup_{s\in[\inf f,\sup f]}\| f -s\cdot 1\|_\infty\\\nonumber
 &=& 2\inf_{t\in\RR}\| f -t\cdot 1\|_\infty
\end{eqnarray*}
See, again, \cite{LSS-arxiv} for details. Now, we can easily deduce the
following optimal Poincar\'e inequality, where we
write $\cP(X)$ for the set of all probability measures on $X$.
\begin{theorem}
 Let $(X,b)$ satisfy {\rm (GVPI)}. Then the best possible   $C_P$ in {\rm
(GVPI)} satisfies:
 $$
 \frac{4}{C_P}=\inf\set{\lae(H(X,b,m))}{m\in \cP(X)\mbox{  \rm{s.t}
}\supp(m)=X} .
 $$
\end{theorem}
Hence we have:
\begin{corollary}
  Let $(X,b)$ be a graph and $m:X\to(0,\infty)$ such that $\vol(X)<\infty$. Then
  $$
  \lae(H(X,b,m))\geq \frac{4}{\diam_r(X)\cdot\vol(X)}
  $$
  and the estimate is optimal.
\end{corollary}
It is not too hard to see that $d=r$ in case that $(X,b)$ is a
locally finite tree, see \cite{GHKLW},  which means that we cannot
do better than in Theorem \ref{lb-fin-N}.

\section{Appendix: Forms and associated
operators, domains and spectra}\label{no-fear} In this section we
present the operator theoretic background for our discussion.

We first mention that
$$
\en(f):=
\frac12 \sum_{x,y\in
X}b(x,y)(f(x)-f(y))^2
$$
defines a closed form on its maximal domain
$$
\dom(\en)={\cal D}\cap \ell^2(X,m)=\{ f\in \ell^2(X,m)\mid
\en(f)<\infty\} .
$$
In fact, $\en$ can be regarded as the sum of the bounded forms
$\en_{x,y}(f)=\frac{1}{2}b(x,y)(f(x)-f(y))^2$ and so is lower
semicontinuous; an appeal to \cite{RS-1}, Theorem S.18 gives the
closedness.

Of course, in the general situation considered here, the domain
$\dom(\en)$ is not necessarily dense in $\ell^2(X,m)$. However, for
$$
\cal{H}_{\en}:=\overline{\dom(\en)},
$$
the closure in $\ell^2(X,m)$, is a Hilbert space and $\left( \en,
\dom(\en)\right)$ defines a closed, densely defined form on
$\cal{H}_{\en}$. By the form representation theorem, Thm VIII.15 in
\cite{RS-1}, there is a unique self adjoint operator $H=H^N(X,b,m)$
in $\cal{H}_{\en}$ that is associated to this form. In analogy with
the continuum situation we call this operator Neumann Laplacian as
it is associated with the maximal form.

\begin{prop}
 Let $(X,b,m)$ be as above and, additionally, $m(X)<\infty$. Then
 \begin{itemize}
  \item [\rm{(a)}]  The function $1$ belongs to $\dom(H)$ with $H 1
  = 0$ and $0$ is  an eigenvalue of multiplicity
one,
  $$
  \inf \sigma(H)=0 .
  $$
  \item [\rm{(b)}]
  $$
  \lae =\inf \sigma(H)\setminus \{0\} .
  $$
  If, furthermore $\diam(X)<\infty$, the latter is an eigenvalue.
 \end{itemize}
\end{prop}
\begin{proof}
 Ad (a): Since $m(X)$ is finite, $1\in\den$ and $\en(1)=0$. Therefore, $H1=0$.
As,
by our standing assumption,
 $(X,b,m)$ is connected, any element in the kernel of $H$ has to be constant, so
the multiplicity of the eigenvalue $0$ is one. It is the bottom of the
spectrum, since $H\ge 0$.

 Ad (b): This follows from the min-max principle, see, e.g., Theorem XIII.2  in
\cite{RS-4}.
 In case that the diameter is finite, the graph is canonically compactifiable
according to \cite{GHKLW},
 see also \cite{LSS-TPI},
 and hence $H$ has compact resolvent, so $\lae$ is the first non-zero
eigenvalue in this case.
\end{proof}

Many of our results deal with the \emph{Dirichlet Laplacian} $H_\Omega$, where
$\Omega\subset X$ is a subset and we imagine the Dirichlet boundary
condition on $D:=X\setminus \Omega$ as given by an infinite potential barrier.
Therefore we get the form
$$
\en_\Omega(\cdot,\cdot)=\en(\cdot,\cdot)\mbox{  on
}\dom(\en_\Omega)=\{ f\in\dom(\en)  \mid f=0\mbox{ on
}D\} .
$$

 We identify $\ell^2(\Omega,m)$ with $\{ f\in\ell^2(X,m)\mid
f=0\mbox{ on }D\}$ and get an associated selfadjoint operator
$H_\Omega$ living in a subspace of $\ell^2(\Omega,m)$ in analogy to
what we saw for the Neumann Laplacian  above. Note that $\en_\Omega$
and $H_\Omega$ are always to be understood relative to the bigger
ambient graph $X$.

\begin{prop}
 Let $(X,b,m)$ be as above and $\Omega\subset X$. Then
$$
\lao = \inf \sigma (H_\Omega) .
$$
 \end{prop}
This, again, is a consequence of the min-max principle.

\begin{rem}  Another natural choice for the relevant forms would be to
consider what might be thought of as Dirichlet boundary conditions
at infinity, given by the form domain
$$
\dom\left(\en^{D,\infty}\right):=\overline{\{ f\in {\cal D}\mid
\supp(f)\mbox{ is a finite set}\}}^\en ,
$$
where the support  of f is given by $\supp (f):=\{ x\in X\mid
f(x)\not=0\}$ and the closure is meant with respect to the form norm
given by the energy. Then, the restriction
$$
\en^{D,\infty}:=\en\rvert_{\dom\left(\en^{D,\infty}\right)}
$$
is a closed form. We do not study the associated operator
$H^{D,\infty}$ here but only mention that it is bounded below by the
Neumann Laplacian. So, our estimates hold for this operator as well.
Similar consideration apply to the restriction to $\Omega$.
\end{rem}

\end{document}